\documentclass[SIADSpaper,final,onefignum,onetabnum]{siamonline181217}

\usepackage{lipsum}
\usepackage{amsfonts}
\usepackage{graphicx}
\usepackage{epstopdf}
\usepackage{algorithmic}
\ifpdf
  \DeclareGraphicsExtensions{.eps,.pdf,.png,.jpg}
\else
  \DeclareGraphicsExtensions{.eps}
\fi

\usepackage{enumitem}
\setlist[enumerate]{leftmargin=.5in}
\setlist[itemize]{leftmargin=.5in}

\newsiamremark{remark}{Remark}
\newsiamremark{hypothesis}{Hypothesis}
\crefname{hypothesis}{Hypothesis}{Hypotheses}
\newsiamthm{claim}{Claim}

\headers{Well-posedness for PAT with Fabry--Perot Sensors}{S. Acosta}

\title{Well-posedness for Photoacoustic Tomography with Fabry--Perot Sensors\thanks{Submitted for review. \funding{This work was partially funded by NSF grant DMS-1712725}}}

\author{Sebastian Acosta\thanks{Predictive Analytics Lab, Baylor College of Medicine and Texas Children's Hospital
  (\email{sebastian.acosta@bcm.edu}, \url{https://sites.google.com/site/acostasebastian01/}).}}

\usepackage{amsopn}

\ifpdf
\hypersetup{
  pdftitle={Solvability of PAT with Fabry--Perot Sensors},
  pdfauthor={S. Acosta}
}
\fi

\usepackage{amsmath}
\usepackage{subcaption}

\newcommand{\Psen}{p_{\rm s}}
\newcommand{\Psub}{p_{\rm b}}

\newcommand{\Usen}{\textbf{u}_{\rm s}}

\newcommand{\rhosen}{\rho_{\rm s}}
\newcommand{\rhosub}{\rho_{\rm b}}

\newcommand{\Csen}{c_{\rm s}}
\newcommand{\Csub}{c_{\rm b}}

\newcommand{\Lambdasub}{\Lambda_{\rm b}}
\newcommand{\Lambdasubh}{\Lambda_{{\rm b},h}}
\newcommand{\Meas}{\mathfrak{m}}
\newcommand{\HH}{\mathcal{H}}
\newcommand{\FF}{\mathcal{F}}
\newcommand{\KK}{\mathcal{K}}

\newsiamremark{assumption}{Assumption}
\crefname{assumption}{Assumption}{Assumptions}

\begin{document}

\maketitle

% REQUIRED
\begin{abstract}
In the mathematical analysis of photoacoustic imaging, it is usually assumed that the acoustic pressure (Dirichlet data) is measured on a detection surface. However, actual ultrasound detectors gather data of a different type. In this paper, we propose a more realistic mathematical model of ultrasound measurements acquired by the Fabry--Perot sensor. This modeling incorporates directional response of such sensors. We study the solvability of the resulting photoacoustic tomography problem, concluding that the problem is well-posed under certain assumptions. Numerical reconstructions are implemented using the Landweber iterations, after discretization of the governing equations using the finite element method. 
\end{abstract}

% REQUIRED

% REQUIRED

%35L05 Partial differential equations: wave equation
%35R30 Partial differential equations: Inverse problems
%35R01 Partial differential equations: on manifolds
%92C55 Biology and other natural sciences: Biomedical imaging and signal processing

%%%%%%%%%%%%%%%%%%%%%%%%%%%%%%%%%%%%%%%%%%%%%%%%%%%%%%%
%%%%%%%%%%%%%%%%   NEW SECTION   %%%%%%%%%%%%%%%%%%%%%%
%%%%%%%%%%%%%%%%%%%%%%%%%%%%%%%%%%%%%%%%%%%%%%%%%%%%%%%
\section{Introduction}

Photoacoustic tomography (PAT) is a hybrid imaging technique based on the photoacoustic effect, which is the transformation of absorbed electromagnetic energy into pressure waves. This technique takes advantage of the fact that absorption exhibits high-contrast in soft biological tissues and that acoustic waves can be measured with broadband transducers leading to imaging with high-resolution. Therefore, high-contrast and high-resolution can be achieved simultaneously 
\cite{Beard2011,Cox2009a,Cox2012,Pramanik2009,  Wang-Anastasio-2011,Wang-2009, Wang2012,Wang-Wu-2007, Wang2003}.

For qualitative photoacoustic tomography, the goal is to form an image of the initial state of the pressure field using boundary measurements of the transient pressure waves.
Most of the reconstruction methods assume that the actual pressure field (Dirichlet data) can be measured at the boundary \cite{Acosta-Montalto-2015,Acosta-Montalto-2016,Acosta-Palacios-2018,Agranovsky2009,Arridge2016,Cox2009a,Frederick2018,Haltmeier2017b,
Haltmeier2018c,Haltmeier2017,Hristova2008,Kowar2011a,Kuchment2008,Nguyen2009,Nguyen2018,Palacios2016b,Qian2011,Ren2018,
Scherzer2017a,Stefanov2009}. In reality, ultrasound sensors are not able to measure the pressure field directly. Instead, they measure certain combinations of the field and its derivatives. This challenge has been noted in \cite{Xu2004} and investigated by Finch \cite{Finch2005p} and by Zangerl, Moon and Haltmeier \cite{Zangerl2018}.

Fabry--Perot transducers offer an alternative to piezoelectric sensors for ultrasound-based imaging applications \cite{Arridge2016c,Beard1999b,Cox2007c,Guggenheim2017,Sheaff2014,Yoo2015,Zhang2008}. The design consists of a sensing film (10-50 $\mu$m thick) sandwiched between extremely thin optically reflective coatings ($\approx$ 50 nm thick) lying on an optically transparent backing substrate ($\approx$ 2 cm thick). An illustration is shown in \cref{fig:Domain}. An interrogating laser beam is employed to generate reflections from both optically reflective coatings. When an incident pressure wave modulates the thickness of the sensing material, the change in the interference pattern from the reflected laser beam is used to estimate the distance between the reflective coatings. The deformation of the sensing material can then be related to pressure measurements. Cox and Beard provide a description of the Fabry--Perot design, and an excellent study of its frequency and directional responses \cite{Cox2007c}. 

\begin{figure}[htbp]
  \centering
  \includegraphics[width=0.45\textwidth]{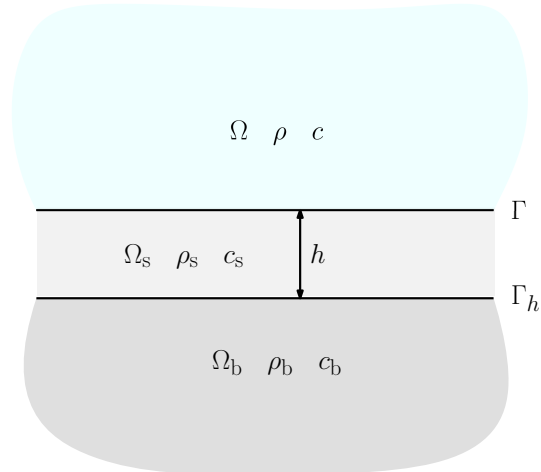}
  \caption{Diagram of domains and boundaries. Acoustic domain $\Omega$ with density $\rho$ and wave speed $c$. Sensing film $\Omega_{\rm s}$ of thickness $h >0$, density $\rhosen$ and wave speed $\Csen$. Backing substrate $\Omega_{\rm b}$ with density $\rhosub$ and wave speed $\Csub$. The interface between the acoustic domain and the sensing film is denoted $\Gamma$. The interface between the sensing film and the backing substrate is denoted $\Gamma_{h}$.}
  \label{fig:Domain}
\end{figure} 

In this paper, we model and investigate the mathematical solvability of the PAT problem for measurements acquired by sensors based on the Fabry--Perot design. We model idealized point-like ultrasound transducers and the physical variable being measured by these sensors. Our goal is to determine whether such measurements lead to the mathematical solvability of the PAT problem. We shall not account for resolution limitations of the Fabry--Perot design due to finite-size sensing elements. We refer the reader to \cite{Cox2007c,Wang2011a,Xia2014,Xu2003c} for investigations concerning this issue. Our modeling is further simplified by ignoring shear waves that can travel in the sensing material and its backing substrate. In other words, we develop an analysis based entirely on the scalar wave equation. 

The acoustic domain $\Omega$ contains soft tissue with density $\rho$ and wave speed $c$. The sensing film $\Omega_{\rm s}$ of thickness $h >0$ has density $\rhosen$ and wave speed $\Csen$. The backing substrate $\Omega_{\rm b}$ has density $\rhosub$ and wave speed $\Csub$. We assume that $\rhosen$, $\rhosub$, $\Csen$ and $\Csub$ are positive constants. However, $\rho$ and $c$ may vary within $\Omega$. The interface between the acoustic domain and the sensing film is denoted $\Gamma$. The interface between the sensing film and the backing substrate is denoted $\Gamma_{h}$. Typically, the sensing film and the backing substrate are acoustically more rigid than the biological soft tissue of interest. Hence, the presence of the sensors induces partial reflections of the waves. Other researchers have investigated PAT with reflecting boundaries assuming that the actual pressure can be measured \cite{Acosta-Montalto-2015,Cox2007,Ellwood2014,Huang2013a,Kunyansky2013,Nguyen2016}. Here we seek to incorporate in our model the influence that the sensor exerts on the pressure waves, as well as the nature of the acoustic measurements for the Fabry--Perot design. The interplay between the sensors and the pressure field is modeled by the following transmission conditions at the interface $\Gamma$,
\begin{equation} \label{eq:Transmission1}
\begin{aligned} 
p = \Psen \quad \text{and} \quad \frac{1}{\rho} \frac{\partial p}{\partial n} = \frac{1}{\rhosen} \frac{\partial \Psen}{\partial n}  \qquad \text{on $\Gamma$,}
\end{aligned}
\end{equation}
where $p$ and $\Psen$ are the pressure in the acoustic medium and sensing material, respectively. The first condition in (\ref{eq:Transmission1}), known as dynamic transmission, ensures the continuity of the pressure field. The second condition in (\ref{eq:Transmission1}), known as kinematic transmission, ensures the continuity of particle motion in the normal direction. Similar transmission conditions hold at the interface $\Gamma_{h}$,
\begin{equation} \label{eq:Transmission2}
\begin{aligned} 
\Psen = \Psub \quad \text{and} \quad \frac{1}{\rhosen} \frac{\partial \Psen}{\partial n} = \frac{1}{\rhosub} \frac{\partial \Psub}{\partial n}  \qquad \text{on $\Gamma_{h}$,}
\end{aligned}
\end{equation}
where $\Psub$ is the pressure in the backing substrate.

In order to simply the analysis, in \Cref{sec:eff_bdy_condition} we derive an effective boundary condition (valid for small $h>0$) to replace the transmission conditions (\ref{eq:Transmission1})-(\ref{eq:Transmission2}). In \Cref{sec:modeling_measurements} we mathematically model the measurements acquired by ultrasound sensors based on the Fabry--Perot design. For the effective boundary condition and modeled boundary measurements, in \Cref{sec:main} we state and prove the solvability of the photoacoustic tomography problem. A reconstruction algorithm is proposed in \Cref{sec:alg} where some numerical experiments are presented as well. The conclusions follow in \Cref{sec:conclusions}.

%%%%%%%%%%%%%%%%%%%%%%%%%%%%%%%%%%%%%%%%%%%%%%%%%%%%%%%
%%%%%%%%%%%%%%%%   NEW SECTION   %%%%%%%%%%%%%%%%%%%%%%
%%%%%%%%%%%%%%%%%%%%%%%%%%%%%%%%%%%%%%%%%%%%%%%%%%%%%%%
\section{Effective boundary condition}
\label{sec:eff_bdy_condition}

For analytical and numerical purposes, it is convenient to replace the transmission conditions (\ref{eq:Transmission1})-(\ref{eq:Transmission2}) for an asymptotically equivalent boundary condition for the acoustic pressure field at the boundary $\Gamma$. This condition is meant to account for the transmission into the sensing film $\Omega_{\rm s}$ and into the backing substrate $\Omega_{\rm b}$ without having to explicitly solve for the wave fields in those domains. See \cite{Antoine2005g,Antoine2005h,Bendali1996,Bonnet2016,Johansson2005,Peron2014} where similar problems are treated. This effective boundary condition also simplifies the model for the measurements as shown in \Cref{sec:modeling_measurements}.

We proceed by making some geometric assumptions about the domain $\Omega_{\rm s}$ occupied by the sensing film. We use the concept of \textit{parallel surfaces} to define the shape of this extremely thin layer of material. These surfaces are parametrized by $0 < r < h$ and defined by $\Gamma_{r} = \left \{ \textbf{y} = \textbf{x} + r \textbf{n}(\textbf{x}) : \textbf{x} \in \Gamma \right \}$. For smooth $\Gamma$ and sufficiently small $h$, each surface $\Gamma_{r}$ is well-defined and smooth. Moreover, the normal vector $\textbf{n}(\textbf{x} + r \textbf{n}(\textbf{x}))$ of the parallel surface $\Gamma_{r}$ coincides with the normal vector $\textbf{n}(\textbf{x})$ of $\Gamma$ for each $\textbf{x} \in \Gamma$. We let $\Omega_{\rm s}$ be the union of this family of parallel surfaces, where $h$ being sufficiently small ensures that each point $\textbf{y} \in \Omega_{\rm s}$ can be uniquely represented in the form $\textbf{y} = \textbf{x} + r \textbf{n}(\textbf{x})$ for $\textbf{x} \in \Gamma$ and $0 < r < h$. See details in \cite[\S 6.2]{Kress-Book-1999} and \cite[Probl. 11 \S 3.5]{DoCarmo1976}. 

The pressure field $\Psen$ in the sensing material satisfies the wave equation,
\begin{equation} \label{eq:wave_eq_sensing}
\Csen^{-2} \frac{\partial^2 \Psen}{\partial t^2} = \frac{\partial^2 \Psen}{\partial n^2} + 2 \HH_{r} \frac{\partial \Psen}{\partial n} + \Delta_{\Gamma_{r}} \Psen, \qquad \text{in $(0,T) \times \Omega_{\rm s}.$}
\end{equation}
For convenience, we have expressed the Laplacian in $\Omega_{\rm s}$ using the normal derivative $\partial/\partial n$ (which makes sense at any point in $\Omega_{s}$ given its definition in terms of parallel surfaces), the mean curvature $\HH_{r}$ of $\Gamma_{r}$ and the Laplace--Beltrami operator $\Delta_{\Gamma_{r}}$ associated with $\Gamma_{r}$. See details in \cite{Antoine1999}. As in \cite{Cox2007c}, we assume that the pressure field $\Psub$ in the backing substrate is outgoing. Therefore, the pressure field $\Psub$ satisfied the following radiation condition
\begin{equation} \label{eq:radiation}
\frac{\partial \Psub}{\partial n} = \Lambdasubh \Psub \qquad \text{on $\Gamma_{h}$}
\end{equation}
where $\Lambdasubh$ is a nonreflecting boundary operator. The subject of nonreflecting or absorbing boundary conditions is beyond the scope of this paper. We refer to \cite{Antoine1999,Barucq2012,Antoine2008,Chniti2016,Acosta2017c} for some relevant articles on that topic. We consider
\begin{equation} \label{eq:DtN}
\Lambdasubh = - \Csub^{-1} \partial_{t} - \HH_{h}
\end{equation}
which is derived in \cite{Antoine1999} as a first order nonreflecting condition that takes into account the mean curvature $\HH_{h}$ of the boundary $\Gamma_{h}$. As shown in \cite[\S 6.2]{Kress-Book-1999} or \cite[Probl. 11 \S 3.5]{DoCarmo1976}, the mean curvatures $\HH$ and $\HH_{h}$ of the surfaces $\Gamma$ and $\Gamma_{h}$, respectively, are related by
\begin{equation*} \label{eq:Curvatures}
\HH_{h} = \frac{ \HH + h\KK }{ 1 + 2 h \HH + h^2 \KK} = \HH + h \left( \KK - 2 \HH^2 \right) + \mathcal{O}(h^2)
\end{equation*}
where $\KK$ is the Gaussian curvature of $\Gamma$. Notice that we are using the mean curvature sign convention from \cite{Antoine1999}, not from \cite{Kress-Book-1999} or \cite{DoCarmo1976}. Therefore, we have that at the surface $\Gamma$, the associated nonreflecting operator $\Lambdasub$ given by 
\begin{equation} \label{eq:DtN2}
\Lambdasub = - \Csub^{-1} \partial_{t} - \HH
\end{equation}
satisfies $\Lambdasubh = \Lambdasub + h \left( \KK - 2 \HH^2 \right) + \mathcal{O}(h^2)$ where $\Lambdasubh$ is defined in (\ref{eq:DtN}).

We proceed with a Taylor expansion for the normal derivative of the pressure field,
\begin{equation*} \label{eq:second_order_approx_Neumann_data}
\begin{aligned} 
\frac{1}{\rho} \left. \frac{\partial p}{\partial n} \right|_{\Gamma} &=  \frac{1}{\rhosen} \left. \frac{\partial \Psen}{\partial n} \right|_{\Gamma} = \frac{1}{\rhosen} \left[ \left. \frac{\partial \Psen}{\partial n} \right|_{\Gamma_{h}} - h \left. \frac{\partial^2 \Psen}{\partial n^2} \right|_{\Gamma}  \right] +  \mathcal{O}(h^2)  \\ 
&= \frac{1}{\rhosen} \left[ \frac{\rhosen}{\rhosub}  \Lambdasubh \left. \left( \Psen + h  \frac{\partial \Psen}{\partial n} \right) \right|_{\Gamma}  - h \left. \frac{\partial^2 \Psen}{\partial n^2} \right|_{\Gamma}   \right] + \mathcal{O}(h^2) \\
&= \frac{1}{\rhosen} \left[ \frac{\rhosen}{\rhosub}  \Lambdasubh \left. \left( p + h \frac{\rhosen}{\rho} \frac{\partial p}{\partial n} \right) \right|_{\Gamma}  - h \left. \left( \Csen^{-2} \partial^{2}_{t} p - 2 \HH \frac{\rhosen}{\rho} \frac{\partial p}{\partial n} - \Delta_{\Gamma} p \right)  \right|_{\Gamma}   \right] + \mathcal{O}(h^2) 
\end{aligned}
\end{equation*}
where we have employed the transmission conditions (\ref{eq:Transmission1})-(\ref{eq:Transmission2}), the wave equation (\ref{eq:wave_eq_sensing}) and the radiation condition (\ref{eq:radiation}). Using (\ref{eq:DtN2}), re-grouping terms and neglecting $\mathcal{O}(h^2)$ terms, we obtain a first order boundary condition
\begin{equation} \label{eq:Eff_bdy_condition}
\frac{1}{\rho} \frac{\partial p}{\partial n} - \frac{1}{\rhosub} \Lambdasub  p = h \left[ \frac{\rhosen}{\rho \rhosub} \Lambdasub \frac{\partial p}{\partial n} + \frac{1}{\rhosub} \left( \KK - 2 \HH^2 \right) p
   -  \frac{1}{\rhosen} \left( \Csen^{-2} \partial_{t}^{2} p - 2 \HH \frac{\rhosen}{\rho} \frac{\partial p}{\partial n} - \Delta_{\Gamma} p \right) \right]
\end{equation}
for the acoustic wave field $p$ on the surface $\Gamma$. This conditions implies that most of the influence (zeroth order terms) that the sensor exerts on the pressure field at the boundary is provided by the thick backing substrate which reflects and refracts the wave field according to the mismatch in densities $\rho$ and $\rhosub$, and in wave speeds $c$ and $\Csub$. The presence of the thin sensing film is accounted for by the first order terms in (\ref{eq:Eff_bdy_condition}). If these latter terms are neglected, the pressure field satisfies the following zeroth order effective boundary condition
\begin{equation} \label{eq:Eff_bdy_condition0}
\frac{1}{\rho}\frac{\partial p}{\partial n} = \frac{1}{\rhosub} \Lambdasub p  \qquad \text{on $\Gamma$}.
\end{equation}

However, the terms of order $\mathcal{O}(h)$ in (\ref{eq:Eff_bdy_condition}) become important in \Cref{sec:modeling_measurements} where we model the ultrasound measurements which are of order $\mathcal{O}(h)$.

%%%%%%%%%%%%%%%%%%%%%%%%%%%%%%%%%%%%%%%%%%%%%%%%%%%%%%%
%%%%%%%%%%%%%%%%   NEW SECTION   %%%%%%%%%%%%%%%%%%%%%%
%%%%%%%%%%%%%%%%%%%%%%%%%%%%%%%%%%%%%%%%%%%%%%%%%%%%%%%
\section{Modeling ultrasound measurements}
\label{sec:modeling_measurements}

For an ultrasound sensor based on the Fabry--Perot design, the quantity being measured is proportional to the difference in the normal projection of the particle displacement on both sides of the sensing film \cite{Cox2007c}. Hence, up to a constant of proportionality, the measurements $\Meas$ acquired by the ultrasound transducer satisfy
\begin{equation}\label{eq:main_measurement}
\partial_{t}^{2} \Meas \sim \left( \left. \partial_{t}^{2} \Usen \right|_{\Gamma}  - \left. \partial_{t}^{2} \Usen \right|_{\Gamma_{h}}  \right)  \cdot \textbf{n} = \frac{1}{\rhosen} \left( \left. \frac{\partial \Psen}{\partial n} \right|_{\Gamma_{h}} - \left. \frac{\partial \Psen}{\partial n} \right|_{\Gamma} \right),
\end{equation}
where the symbol $\sim$ means equality up to a multiplicative constant, and the pressure--displacement formulation is valid in the absence of shear stress. We seek to express the measurement in terms of the pressure in the acoustic medium only. Using the transmission conditions (\ref{eq:Transmission1})-(\ref{eq:Transmission2}) at both sides of the sensing film, and (\ref{eq:radiation})-(\ref{eq:DtN2}), we obtain
\begin{equation} \label{eq:transmission_aux}
\begin{aligned} 
 \frac{\partial \Psen}{\partial n} &= \frac{\rhosen}{\rho} \frac{\partial p}{\partial n} \quad \text{on $\Gamma$,}  \\
\frac{\partial \Psen}{\partial n} &= \frac{\rhosen}{\rhosub} \frac{\partial \Psub}{\partial n} = \frac{\rhosen}{\rhosub} \Lambdasubh \Psub = \frac{\rhosen}{\rhosub} \Lambdasubh \Psen = \frac{\rhosen}{\rhosub} \left[\Lambdasub + h\left( \KK - 2 \HH^2 \right)\right] \Psen \quad \text{on $\Gamma_{h}$.}
\end{aligned}
\end{equation}

We make the following Taylor approximation for the pressure field $\Psen$ within the sensing film and combine it with (\ref{eq:Transmission1})-(\ref{eq:Transmission2}) and (\ref{eq:transmission_aux}) to obtain,
\begin{equation} \label{eq:second_order_approx}
\begin{aligned} 
\left. \Psen \right|_{\Gamma_{h}} = \left. \left[ \Psen + h \frac{\partial \Psen}{\partial n} \right] \right|_{\Gamma} + \mathcal{O}(h^2) = \left. \left[ p + h \frac{\rhosen}{\rho} \frac{\partial p}{\partial n} \right] \right|_{\Gamma} + \mathcal{O}(h^2).
\end{aligned}
\end{equation}

Therefore, combining (\ref{eq:main_measurement})-(\ref{eq:second_order_approx}) we obtain an expression for the measurements in terms of the acoustic pressure field (Dirichlet data) and its normal derivative (Neumann data) at the boundary $\Gamma$ as follows,
\begin{equation} \label{eq:eff_measurement}
\partial_{t}^2 \Meas \sim \left( \frac{1}{\rhosub} \Lambdasub p - \frac{1}{\rho} \frac{\partial p}{\partial n} \right) + h  \left[ \frac{\rhosen}{\rho \rhosub} \Lambdasub \frac{\partial p}{\partial n} + \left(\KK - 2\HH^2  \right) p \right] + \mathcal{O}(h^2)
\end{equation}
Neglecting the $\mathcal{O}(h^2)$ terms on the right-hand side of (\ref{eq:eff_measurement}) and using the effective boundary conditions (\ref{eq:Eff_bdy_condition})-(\ref{eq:Eff_bdy_condition0}), we obtain a simplified or first order model for the measurements,
\begin{equation} \label{eq:eff_measurement_simple}
\partial_{t}^2 \Meas \sim  \partial_{t}^{2} p + 2 \HH \frac{\rhosen}{\rhosub} \frac{\Csen^{2}}{\Csub}  \partial_{t} p +  2 \HH^2 \Csen^{2} \frac{\rhosen}{\rhosub} p - \Csen^{2} \Delta_{\Gamma} p   \qquad \text{on $(0,T) \times \Gamma$} 
\end{equation}
where we have used the definition of the operator $\Lambdasub$ given by (\ref{eq:DtN2}). In order to fully determine $\Meas$, initial conditions must be provided. In consistency with the PAT scenario, where the pressure field has vanishing initial Cauchy data in the exterior of $\Omega$, we let $\Meas = \partial_{t} \Meas = 0$ on $\{ t = 0\} \times \Gamma$.

To illustrate the response associated with this sensor design, we briefly analyze its behavior for plane waves. For convenience, we momentarily assume that $\Gamma$ is a plane through the origin. Both, the effective boundary condition (\ref{eq:Eff_bdy_condition0}) and the form of the measurements (\ref{eq:eff_measurement_simple}) play a role in this analysis. From (\ref{eq:DtN2}) we have $\Lambdasub = - \Csub^{-1} \partial_{t}$ because for a flat surface $\Gamma$ the mean curvature is $\HH = 0$. A plane wave $p_{\rm inc} = e^{i\left( \textbf{x} \cdot \textbf{k} - \omega t \right)}$ with incidence angle $\theta$, induces a reflection governed by the effective boundary condition (\ref{eq:Eff_bdy_condition0}). The superposition of the incident and reflected waves has the form
\begin{equation} \label{eq:plane_wave}
p(\textbf{x},t) = e^{i\left( \textbf{x} \cdot \textbf{k} - \omega t \right)} + R e^{i\left( \textbf{x} \cdot \textbf{k}_{\rm r} - \omega t \right)}
\end{equation}
where $R$ is the reflection coefficient, $\textbf{k}_{\rm r}$ is the reflection wavenumber, such that $|\textbf{k}| = |\textbf{k}_{\rm r}| = \omega/c$ and $\textbf{n} \cdot \textbf{k}_{\rm r} = - \textbf{n} \cdot \textbf{k}$ where $\textbf{n}$ is the outward normal vector. We also have $\textbf{n} \cdot \textbf{k} = |\textbf{k}| \cos \theta$. Once (\ref{eq:plane_wave}) is plugged into (\ref{eq:Eff_bdy_condition0}), the reflection coefficient is shown to satisfy
\begin{equation} \label{eq:reflection_coeff}
R = \frac{\cos \theta - \alpha}{ \cos \theta + \alpha}, \qquad \text{where} \quad \alpha = \frac{\rho c }{\rhosub \Csub}.
\end{equation}
Plugging (\ref{eq:plane_wave})-(\ref{eq:reflection_coeff}) into (\ref{eq:eff_measurement_simple}) and evaluating at the origin $\textbf{x} = \textbf{0}$, we find that the measurements satisfy
\begin{equation} \label{eq:meas_plane_wave}
\frac{\Meas}{p_{\rm inc}} \sim \left( 1 + \frac{\cos \theta - \alpha}{ \cos \theta + \alpha} \right) \left( 1 - \frac{\Csen^2}{c^2} \sin^2 \theta \right).
\end{equation}

\begin{figure}[htbp]
  \centering
  \includegraphics[width=0.75\textwidth]{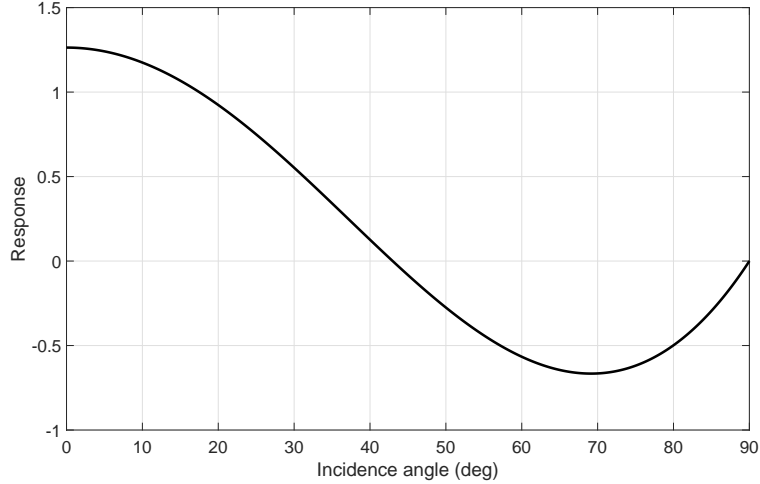}
  \caption{Directional response of Fabry--Perot measurements for plane waves. The parameters, taken from \cite{Cox2007c}, correspond to a  Parylene sensing film with (compressional) wave speed $\Csen = 2200$ m/s and density $\rhosen = 1180$ kg/m$^3$ and a polycarbonate backing substrate with (compressional) wave speed $\Csub = 2180$ m/s and density $\rhosub = 1180$ kg/m$^3$. The acoustic medium corresponds to water with wave speed $c = 1500$ m/s and density $\rho = 1000$ kg/m$^3$.}
  \label{fig:Response}
\end{figure}

\Cref{fig:Response} displays the directional response (\ref{eq:meas_plane_wave}) for plane waves as a function of the incidence angle $\theta$. The parameters were taken from \cite{Cox2007c} for a Parylene sensing film and polycarbonate backing substrate. \Cref{fig:Response} shows that incidence at approximately $42.99^{\circ}$ corresponds to a critical angle where the waves cause no particle motion in the normal direction. This occurs when the tangential phase speed of the acoustic wave equals the wave speed in the sensing film. The Fabry--Perot sensor does not capture such acoustic waves. We also note that for incidence angles greater than this critical angle, the pressure wave and the measurement have opposite signs. The response also approaches zero as the incidence angle approaches $90^{\circ}$.
Therefore, for incidence at tangential angles, the sensor design is not able to measure the pressure waves adequately.

The above are some of the features not taken into account when it is assumed that the actual pressure field (Dirichlet data) is acquired at the measurement boundary. As explained in the Introduction, the overly idealized assumption for ultrasound sensors is that they measure the pressure field at the boundary and that the sensors themselves do not perturb the pressure waves. This overly idealized model can be expressed as follows,
\begin{equation} \label{eq:overly_ideal_measurement}
\Meas \sim  p  \qquad \text{on $(0,T) \times \Gamma$,} \quad \text{and $\rhosen = \rhosub = \rho$ and $\Csen = \Csub = c$.} 
\end{equation}

%%%%%%%%%%%%%%%%%%%%%%%%%%%%%%%%%%%%%%%%%%%%%%%%%%%%%%%
%%%%%%%%%%%%%%%%   NEW SECTION   %%%%%%%%%%%%%%%%%%%%%%
%%%%%%%%%%%%%%%%%%%%%%%%%%%%%%%%%%%%%%%%%%%%%%%%%%%%%%%
\section{Main mathematical results}
\label{sec:main}

Here we define the photoacoustic tomography problem in terms of the wave equation, the effective boundary condition (\ref{eq:Eff_bdy_condition0}) and the ultrasound measurements modeled by (\ref{eq:eff_measurement_simple}). We also prove the solvability of this problem under the following geometric condition for the wave speed $c$ and the domain $\Omega$. 

\begin{assumption}[Non--trapping Condition] \label{assumption:nontrapping}
Let $\Omega$ be a simply-connected bounded domain with smooth boundary $\Gamma$. Assume there exists $T_{\rm o} < \infty$ such that any geodesic ray of the manifold $(\Omega, c^{-2} dx^2)$, originating from any point in $\Omega$ at time $t=0$ reached the boundary $\Gamma$ at a non-diffractive point before $t = T_{\rm o}$. 
\end{assumption}

The forward mapping, which we seek to invert, is given by
\begin{equation} \label{eq:forward_map}
\FF : \left( p_{0},p_{1} \right) \mapsto \Meas
\end{equation}
where the measurement $\Meas$ satisfies (\ref{eq:eff_measurement_simple}) with vanishing initial Cauchy data. The initial velocity $p_{1}$  is known to be zero in the context of PAT. However, the mathematical theory allows us to recover it as well.  The pressure field $p$ solves the following initial boundary value problem,
\begin{equation} \label{eq:IBVP}
\begin{aligned} 
\partial_{t}^2 p &= c^{2} \Delta p \qquad  &&  \text{in $(0,T) \times \Omega$,} \\
\rhosub  \partial_{n} p &= \rho \Lambdasub p  \qquad  &&  \text{on $(0,T) \times \Gamma$,} \\
p = p_{0} \quad &\text{and} \quad \partial_{t} p = - p_{1} \qquad  &&  \text{on $\{ t = 0 \} \times \Omega$.}
\end{aligned}
\end{equation}
The well-posedness of this problem for $\left( p_{0},p_{1} \right) \in H^{1}_{0}(\Omega) \times H^{0}(\Omega)$ has been established. See for instance \cite{EvansPDE,Lasiecka1986,Lio-Mag-Book-1972}. The unique solution satisfies $p \in C([0,T];H^{1}(\Omega))$, $\partial_{t} p \in C([0,T];H^{0}(\Omega))$, $p|_{\Gamma} \in H^{1}((0,T) \times \Gamma)$ and $\partial_{n} p \in H^{0}((0,T)\times \Gamma)$. We work with the standard Sobolev spaces based on square-integrable functions over $\Omega$ or $(0,T) \times \Gamma$. The
associated inner-product extends naturally as the duality pairing between functionals and functions. We should interpret the Hilbert space $H^{0}(\Omega)$ with the inner-product appropriately
weighted by $c^{-2}$ so that $c^2 \Delta$ is formally self-adjoint with respect to the duality
pairing of $H^{0}(\Omega)$.

Now we state our main result in the form of a theorem.

\begin{theorem}[Main Result]  \label{thm:main}
Under the non--trapping \cref{assumption:nontrapping} for the manifold $(\Omega,c^{-2} dx^2)$ and time $T> T_{\rm o}$, the forward mapping $\FF: H^{1}_{0}(\Omega) \times H^{0}(\Omega) \to H^{0}((0,T);H^{1}(\Gamma))$ is injective, that is, the photoacoustic tomography problem is uniquely solvable. Moreover, the following stability estimate holds 
\begin{equation} \label{eq:Stability0}
\| \left( p_{0},p_{1} \right) \|_{H^{0}(\Omega) \times H^{-1}(\Omega)} \leq C \| \Meas \|_{H^{0}((0,T);H^{1}(\Gamma))}
\end{equation}
for some constant $C>0$.
\end{theorem}

We proceed to prove this theorem by showing that the adjoint of the forward mapping is surjective. This adjoint mapping $\FF^{*}$ is given by 
\begin{equation} \label{eq:adjoint_map}
\FF^{*} : \psi \mapsto \left( \partial_{t} \xi(0) , \xi(0) \right)
\end{equation}
where $\xi$ solves the following backwards--in--time boundary value problem,
\begin{equation} \label{eq:IBVP_adj}
\begin{aligned} 
 \partial_{t}^2 \xi &= c^{2} \Delta\xi \qquad  &&  \text{in $(0,T) \times \Omega$,} \\
\rhosub \partial_{n} \xi - \rho \Lambdasub^{*} \xi &= \rhosub \left( \partial_{t}^2 + a \partial_t + b - \Csen^{2}\Delta_{\Gamma} \right)^{*} \left( \partial_{t}^{-2} \right)^{*} \psi \qquad  &&  \text{on $(0,T) \times \Gamma$,} \\
\xi = 0 \quad &\text{and} \quad \partial_{t} \xi = 0 \qquad  &&  \text{on $\{ t = T \} \times \Omega$,}
\end{aligned}
\end{equation}
where $a = 2 \HH  \Csen^2 \rhosen/ ( \Csub \rhosub)$ and $b = 2 \HH^2 \Csen^2 \rhosen / \rhosub$ are constants. The operator $\partial_{t}^{-2}$ can be defined as
\begin{equation} \label{eq:anti-der}
\left( \partial_{t}^{-2} v \right)(t) = \int_{0}^{t} \int_{0}^{\tau} v(s) ds d\tau.
\end{equation}
Notice that $\partial_{t}^{-2} \partial_{t}^{2} v = \partial_{t}^{2} \partial_{t}^{-2} v = v$ for all sufficiently smooth $v$ such that $v|_{t=0} = \partial_{t} v|_{t=0}= 0$. Also notice that $\partial_{t}^{2}$ is formally self--adjoint. In particular, 
\begin{equation} \label{eq:self_adjoint_second_der}
\left( \partial_{t}^{2} v, \phi \right)_{H^{0}((0,T)\times \Gamma)} = \left( v, \partial_{t}^{2} \phi \right)_{H^{0}((0,T)\times \Gamma)}
\end{equation}
for all sufficiently regular $v$ and $\phi$ such that $v|_{t=0} = \partial_{t} v|_{t=0}= 0$ and $\phi|_{t=T} = \partial_{t} \phi|_{t=T}= 0$. The Laplace--Beltrami operator $\Delta_{\Gamma}$ is also self--adjoint since the manifold $\Gamma$ has no boundary. The nonreflecting operator $\Lambdasub$ defined in (\ref{eq:DtN2}) has an adjoint given by $\Lambdasub^{*} = \Csub^{-1} \partial_{t} - \HH$.
Therefore, it stays as a differential operator with first and zeroth order terms.

The statement of \cref{thm:main} is a direct consequence of the following lemma.

\begin{lemma}\label{lemma:surjectivity}
Under the non--trapping \cref{assumption:nontrapping} for the manifold $(\Omega,c^{-2} dx^2)$ and time $T> T_{\rm o}$, the operator $\FF^{*} : H^{0}((0,T); H^{-1}(\Gamma)) \to H^{0}(\Omega) \times H^{1}(\Omega)$ is surjective.
\end{lemma}

\begin{proof}
The mapping $\FF^{*}$ can be composed as $\FF^{*} = \mathcal{G}_{2} \circ \mathcal{G}_{1}$ where
\begin{equation} \label{eq:Components}
\begin{aligned} 
\mathcal{G}_{1} &: \psi \mapsto \varphi, \\
\mathcal{G}_{2} &: \varphi \mapsto \left( \partial_{t} \xi(0) , \xi(0) \right),
\end{aligned}
\end{equation}
where the mapping $\mathcal{G}_{1}$ is given by 
\begin{equation} \label{eq:IBVP_bdy}
\begin{aligned} 
\varphi = \partial_{t}^2 \Psi - a \partial_{t} \Psi + b \Psi - \Csen^{2} \Delta_{\Gamma} \Psi \qquad  &&  \text{in $(0,T) \times \Gamma$}
\end{aligned}
\end{equation}
where $\Psi$ has vanishing Cauchy data at $t=T$ and solves $\partial_{t}^{2} \Psi = \rhosub \psi$. The mapping $\mathcal{G}_{2}$ is defined via $\xi$, the solution of 
\begin{equation} \label{eq:IBVP_adj_aux}
\begin{aligned} 
 \partial_{t}^2 \xi &= c^{2} \Delta\xi \qquad  &&  \text{in $(0,T) \times \Omega$,} \\
\rhosub \partial_{n} \xi - \rho \Lambdasub^{*} \xi &= \varphi \qquad  &&  \text{on $(0,T) \times \Gamma$,} \\
\xi = 0 \quad &\text{and} \quad \partial_{t} \xi = 0 \qquad  &&  \text{on $\{ t = T \} \times \Omega$.}
\end{aligned}
\end{equation}

Under the non--trapping assumption, the mapping $\mathcal{G}_{2} : \varphi \mapsto \left( \partial_{t} \xi(0) , \xi(0) \right)$ defined by (\ref{eq:IBVP_adj_aux}), is well--known to be surjective from $H^{0}((0,T) ; H^{0}(\Gamma))$ onto $H^{0}(\Omega) \times H^{1}(\Omega)$. That is the central theme of exact boundary controllability for the wave equation. See \cite[Ch 6]{GlowinskiLionsHe2008} and \cite[Corollary 4.10]{Bardos1992}.

Hence, it only remains to show that the mapping $\mathcal{G}_{1}$ is surjective. This can be accomplished by proving that  equation (\ref{eq:IBVP_bdy}) is solvable for any forcing term $\varphi \in H^{0}((0,T) ; H^{0}(\Gamma))$ such that $\Psi$ has vanishing Cauchy data at $t=T$ and $\partial_{t}^{2} \Psi \in H^{0}((0,T);H^{-1}(\Gamma))$. This solvability is a well--established result. See \cite[Theorems 3-5 in \S 7.2]{EvansPDE}, \cite{Lasiecka1986} and \cite[Ch 3 \S 8, Thm 8.1]{Lio-Mag-Book-1972}. As a consequence, the operator $\FF^{*}$ is surjective from $H^{0}((0,T); H^{-1}(\Gamma))$ onto $H^{0}(\Omega) \times H^{1}(\Omega)$.
\end{proof}

This \cref{lemma:surjectivity} renders the proof of \cref{thm:main}. Indeed, we first notice that since $\FF^{*} : H^{0}((0,T); H^{-1}(\Gamma)) \to H^{0}(\Omega) \times H^{1}(\Omega)$ is well-defined and surjective, then $\FF: H^{0}(\Omega) \times H^{-1}(\Omega) \to H^{0}((0,T);H^{1}(\Gamma))$ is well-defined, injective, and has a closed range. The stability estimate (\ref{eq:Stability0}) then follows from the Open Mapping Theorem (see \cite[Ch 2]{McLean2000} or 
\cite[Ch 2]{Drabek-Milota-2007}).

%%%%%%%%%%%%%%%%%%%%%%%%%%%%%%%%%%%%%%%%%%%%%%%%%%%%%%%
%%%%%%%%%%%%%%%%   NEW SECTION   %%%%%%%%%%%%%%%%%%%%%%
%%%%%%%%%%%%%%%%%%%%%%%%%%%%%%%%%%%%%%%%%%%%%%%%%%%%%%%
\section{Numerical results}
\label{sec:alg}

In this section we implement reconstruction algorithms to the solve the PAT problem at the discrete level. The reconstructions are based on the Landweber iterative method \cite[Ch. 6]{EnglBook2000}. In the context of PAT, the Landweber iteration has been employed previously because of its simplicity and compatibility with regularization methods \cite{
Belhachmi2016,Ding2015,Haltmeier2017b,Haltmeier2018c,Haltmeier2017,Nguyen2018,Stefanov2017a}. Other approaches, such as the conjugate gradient method, may also be employed to solve PAT problems \cite{Acosta-Montalto-2015,Acosta-Montalto-2016,Haltmeier2017,Modgil2010,Nguyen2018,
Wang2011a,WangHuang2013,Wang2012b,Xu2003}. The Landweber iteration is defined in \cref{alg:landweber}.

\begin{algorithm}
\caption{Landweber iteration}
\label{alg:landweber}
\begin{algorithmic}
\STATE{Initial guess $\phi_{0} = 0$.}
\STATE{Set $0 < \gamma < \| \FF \|^{-2}$.}
\FOR{$k=0,1,2,...$}
\STATE{$\phi_{k+1} = \phi_{k} - \gamma \FF^{*} \left( \FF \phi_{k} - \Meas \right)$}
\ENDFOR
\end{algorithmic}
\end{algorithm}

\begin{figure}[htbp]
  \centering
  \includegraphics[width=0.48\textwidth]{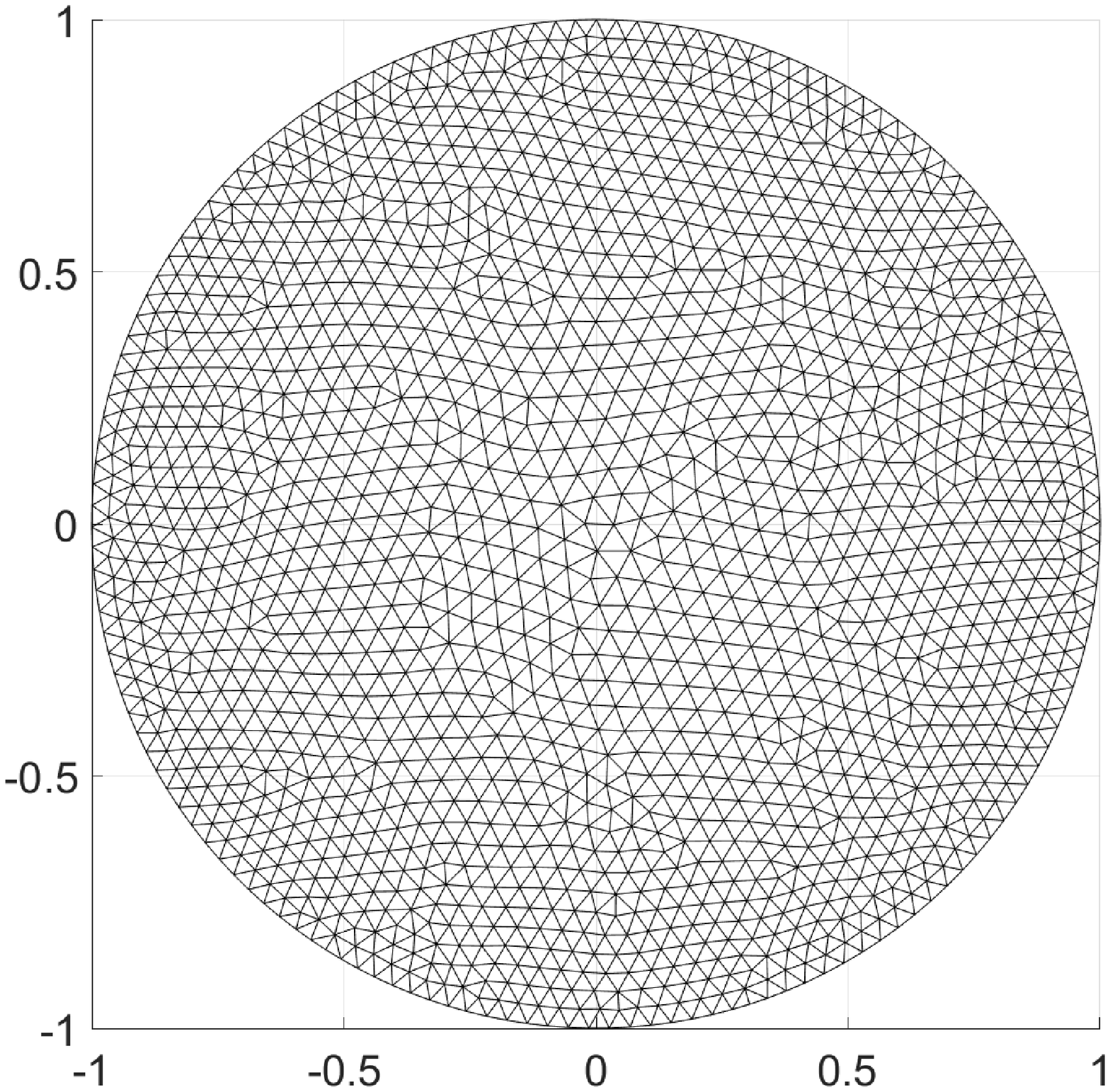}
  \includegraphics[width=0.48\textwidth]{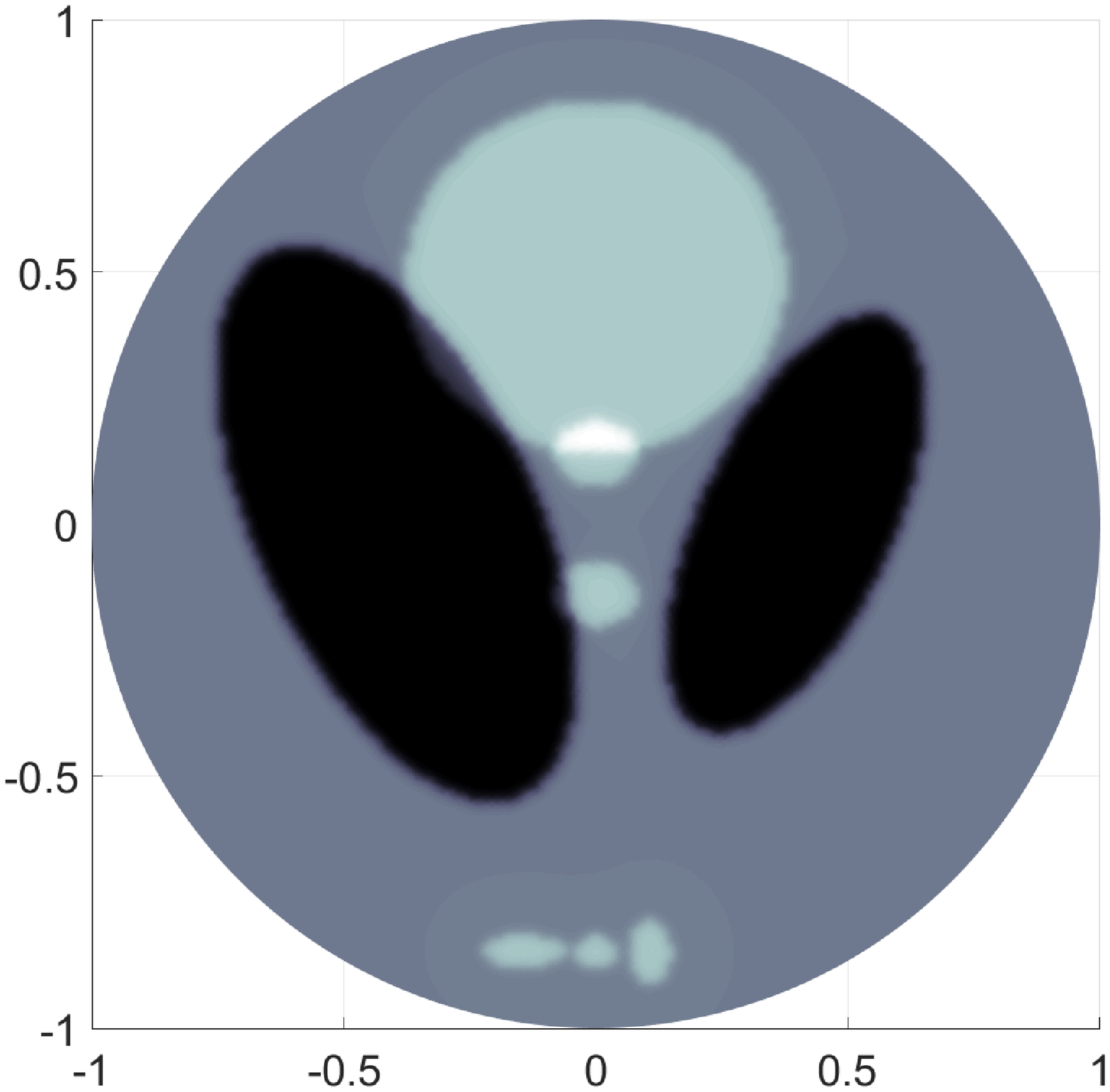}
  \caption{Coarse triangulation for the FEM (left) and exact profile to be reconstructed (right).}
  \label{fig:Exact}
\end{figure}

The discretizations of the forward map (\ref{eq:forward_map}) and adjoint map (\ref{eq:adjoint_map}) are based on a piecewise linear finite element method (FEM) and second order finite difference for the time derivatives in the initial boundary value problems (\ref{eq:IBVP}) and (\ref{eq:IBVP_adj}), respectively. The discretization parameters were chosen to satisfy the CFL stability condition. The FEM was implemented on triangular meshes of the unit disk and the physical parameters, described in \cref{fig:Response}, were non-dimensionalized accordingly. \Cref{fig:Exact} shows a coarse triangular mesh and the Shepp--Logan phantom to be reconstructed. Measured data was synthetically generated by discretizing the forward operator $\FF$ using the FEM method. In all simulations, the mesh employed to generate the measurements had mesh size about half of the mesh size for the mesh employed to reconstruct the phantom.

\begin{figure}[htbp]
  \centering
  \includegraphics[width=0.85\textwidth]{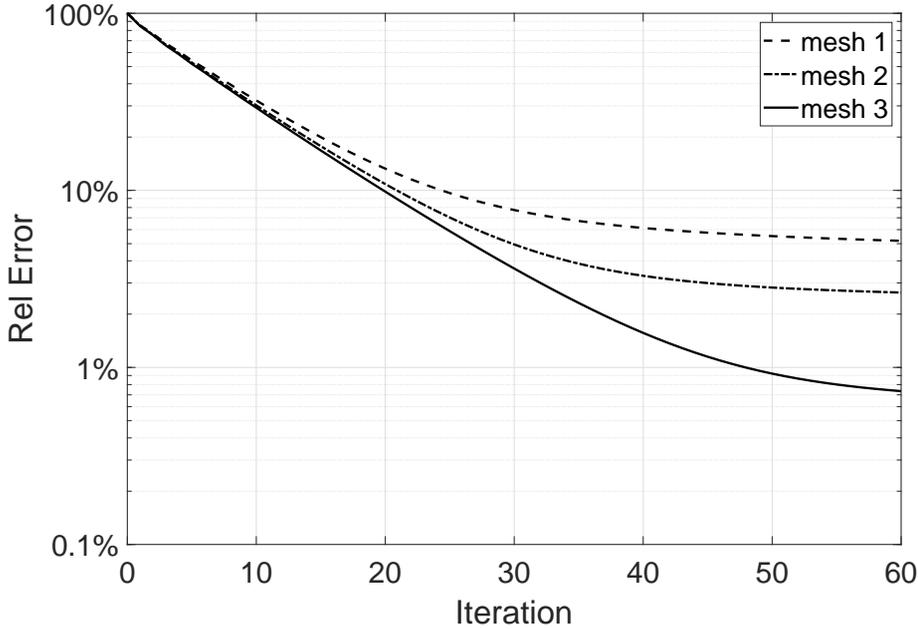}
  \caption{Relative error for the Landweber iterations. The physical parameters, same as in \cref{fig:Response}, were non-dimensionalized for these simulations. The mesh refinement leads to the numbers of degrees of freedom 1713, 6875 and 27248 for the meshes 1, 2 and 3, respectively. This corresponds to halving the mesh size in consecutive refinements.}
  \label{fig:Error}
\end{figure}

\Cref{fig:Error} displays the error $\| \phi_{k} - p_{0} \|$ (where $p_{0}$ is the exact solution) as a function of the iteration number $k \geq 0$ of the Landweber method. The figure shows results for three FEM meshes that were consecutively refined by halving the mesh size. We notice that initially, the error decays exponentially in $k$ (as the theory of this method predicts) but then it stagnates. The stagnation level decreases with mesh refinement. This phenomenon may be attributed to the fact that the discrete version of $\FF^{*}$ is not the actual adjoint of the discrete version of $\FF$. Thus, the discretization of the normal operator $\FF^{*} \FF$ is not symmetric positive definite as this method requires. However, as the mesh is refined, we expect this error to reduce. Implementation of an exact numerical adjoint, as done by Huang \textit{et al.} \cite{Huang2013}, could remedy this issue.

Lastly, we compare the reconstruction of the initial pressure profile obtained by accounting for the  structure of the Fabry--Perot measurement model (\ref{eq:eff_measurement_simple}) against the reconstruction obtained from the overly idealized (but commonly assumed) model (\ref{eq:overly_ideal_measurement}). The latter reconstruction is obtained by synthetically producing the measurements following the model (\ref{eq:eff_measurement_simple}), but then incorrectly assuming that the measurements satisfy (\ref{eq:overly_ideal_measurement}). \Cref{fig:Recon} displays the reconstruction results for both measurement models using 60 iterations of the Landweber method. For the reconstruction following the proposed model (\ref{eq:eff_measurement_simple}), the relative error is $0.73\%$. For the reconstruction following the overly idealized model (\ref{eq:overly_ideal_measurement}), the relative error is $23.02\%$.

\begin{figure}[htbp]
  \centering
  \begin{subfigure}[b]{1.0\textwidth}
  \includegraphics[width=0.48\textwidth]{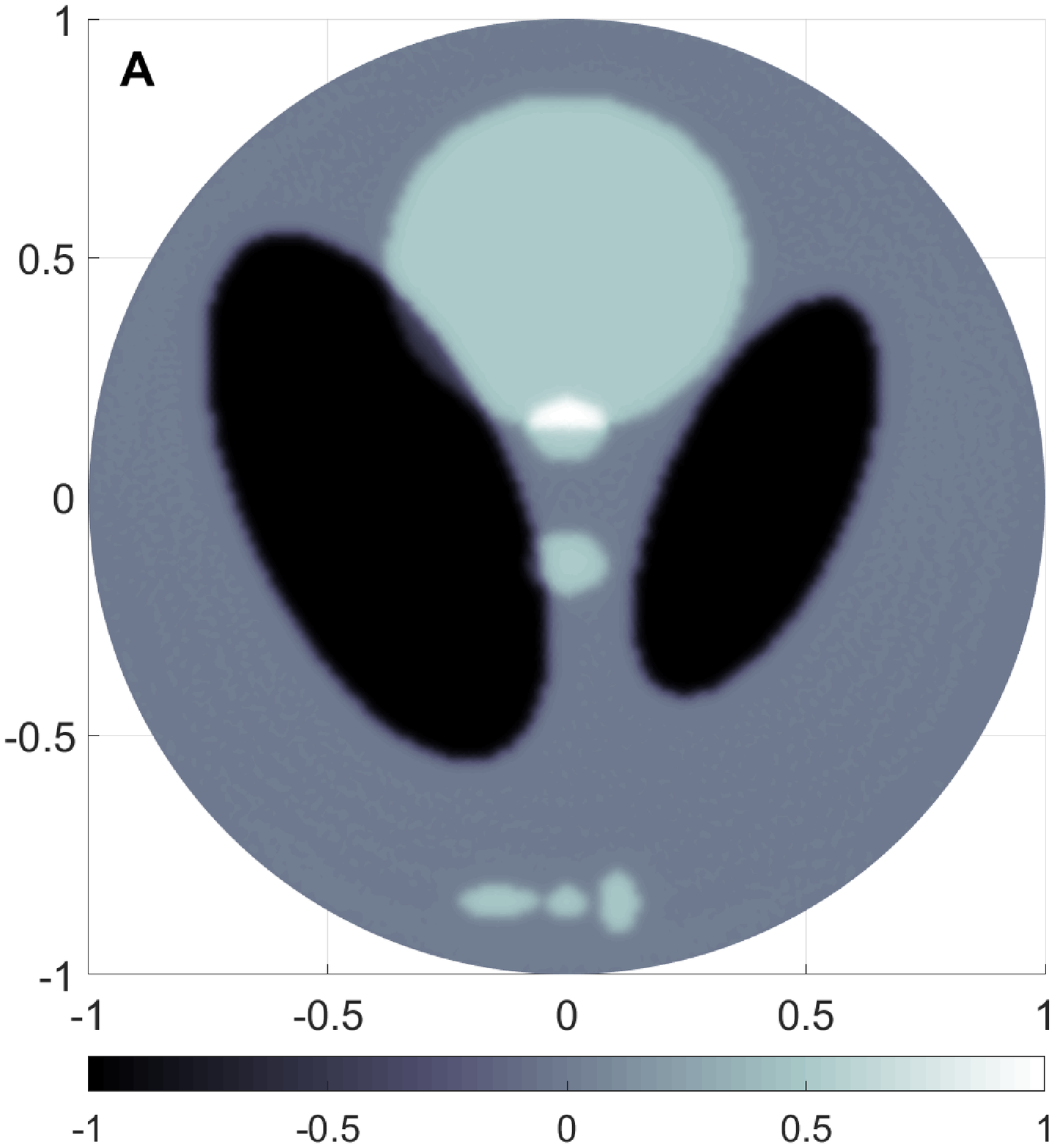}
  \includegraphics[width=0.48\textwidth]{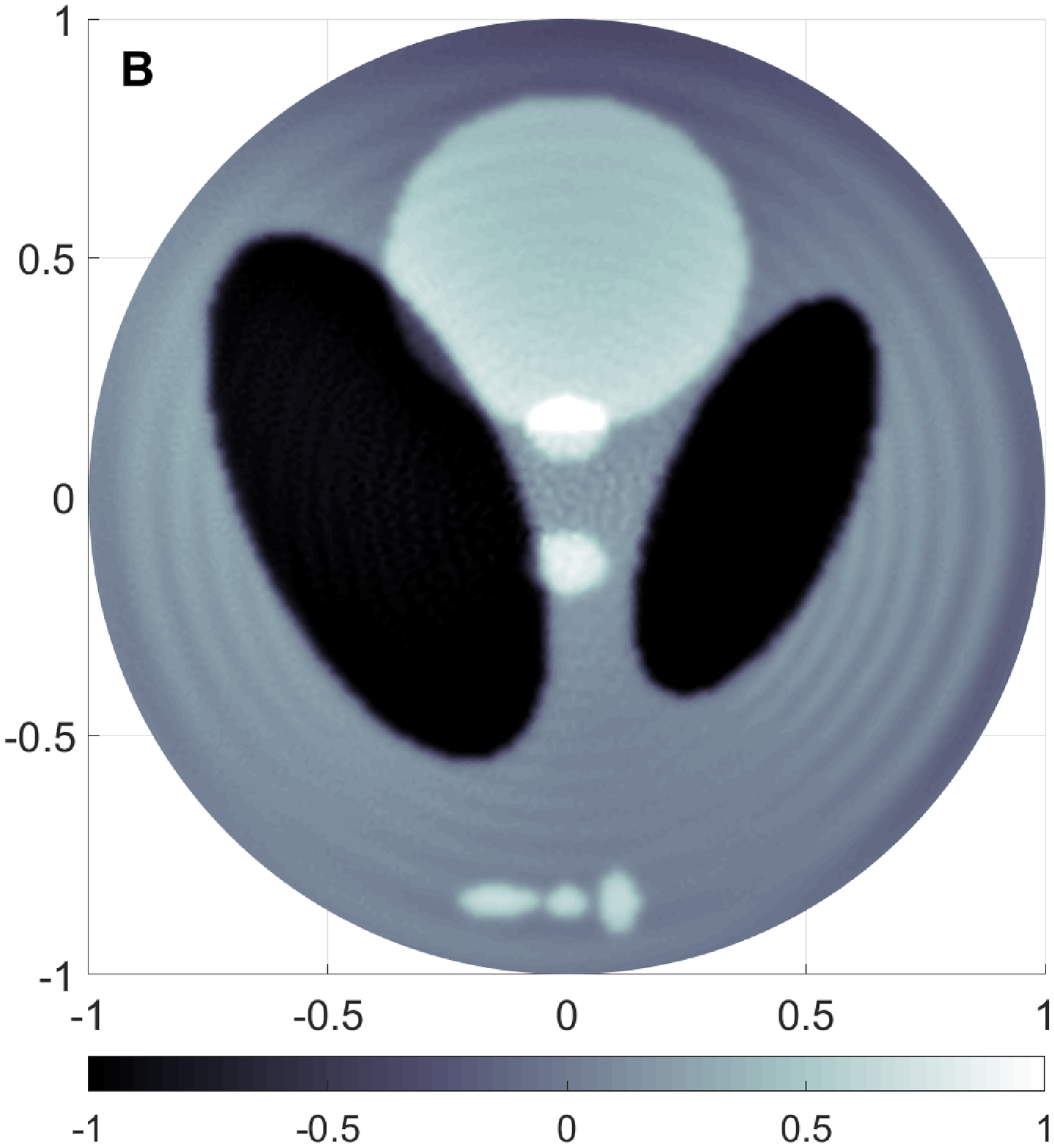}
  \vspace{-1em}
  \end{subfigure} \\
  \begin{subfigure}[b]{1.0\textwidth}
  \includegraphics[width=0.48\textwidth]{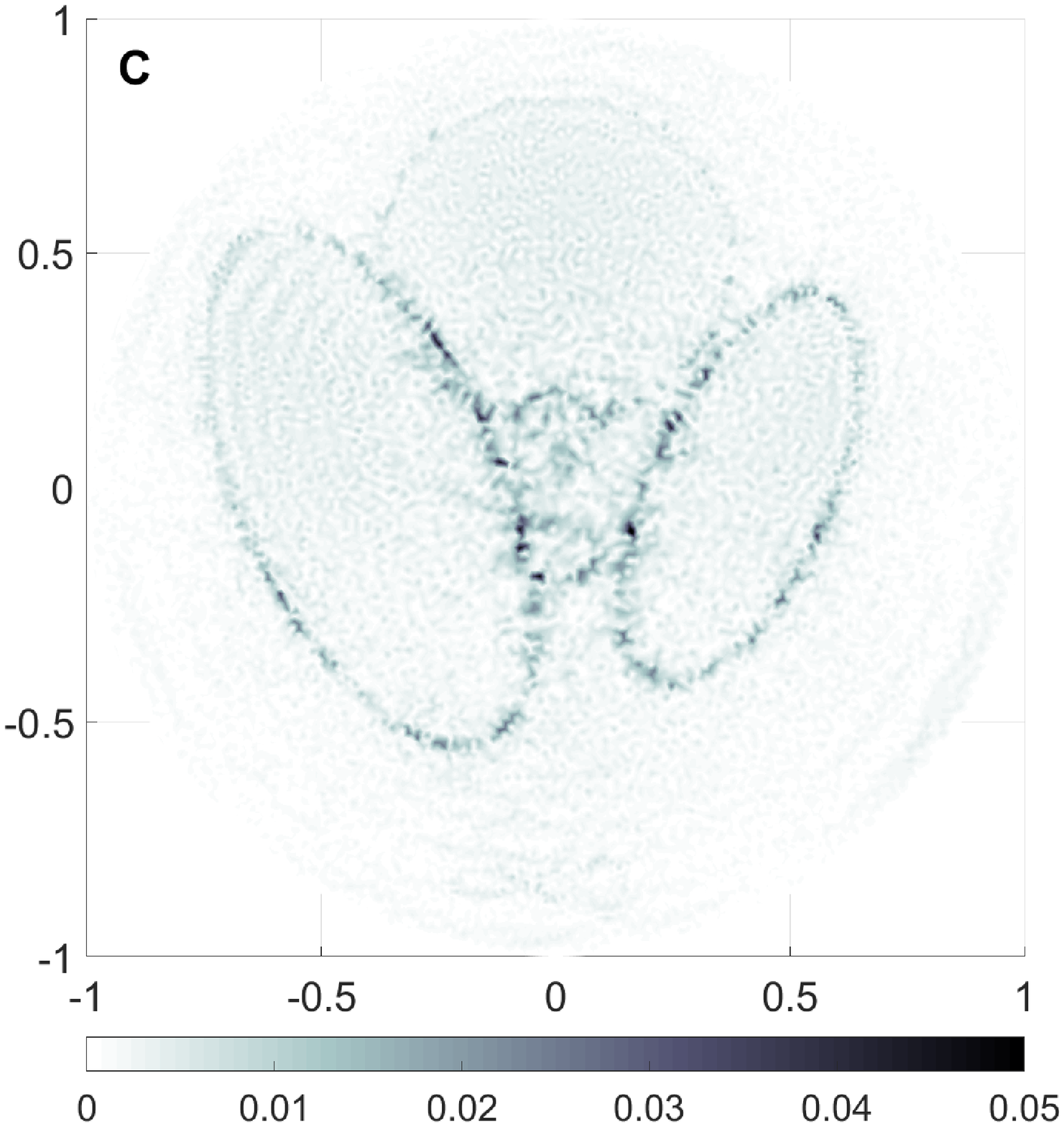}
  \includegraphics[width=0.48\textwidth]{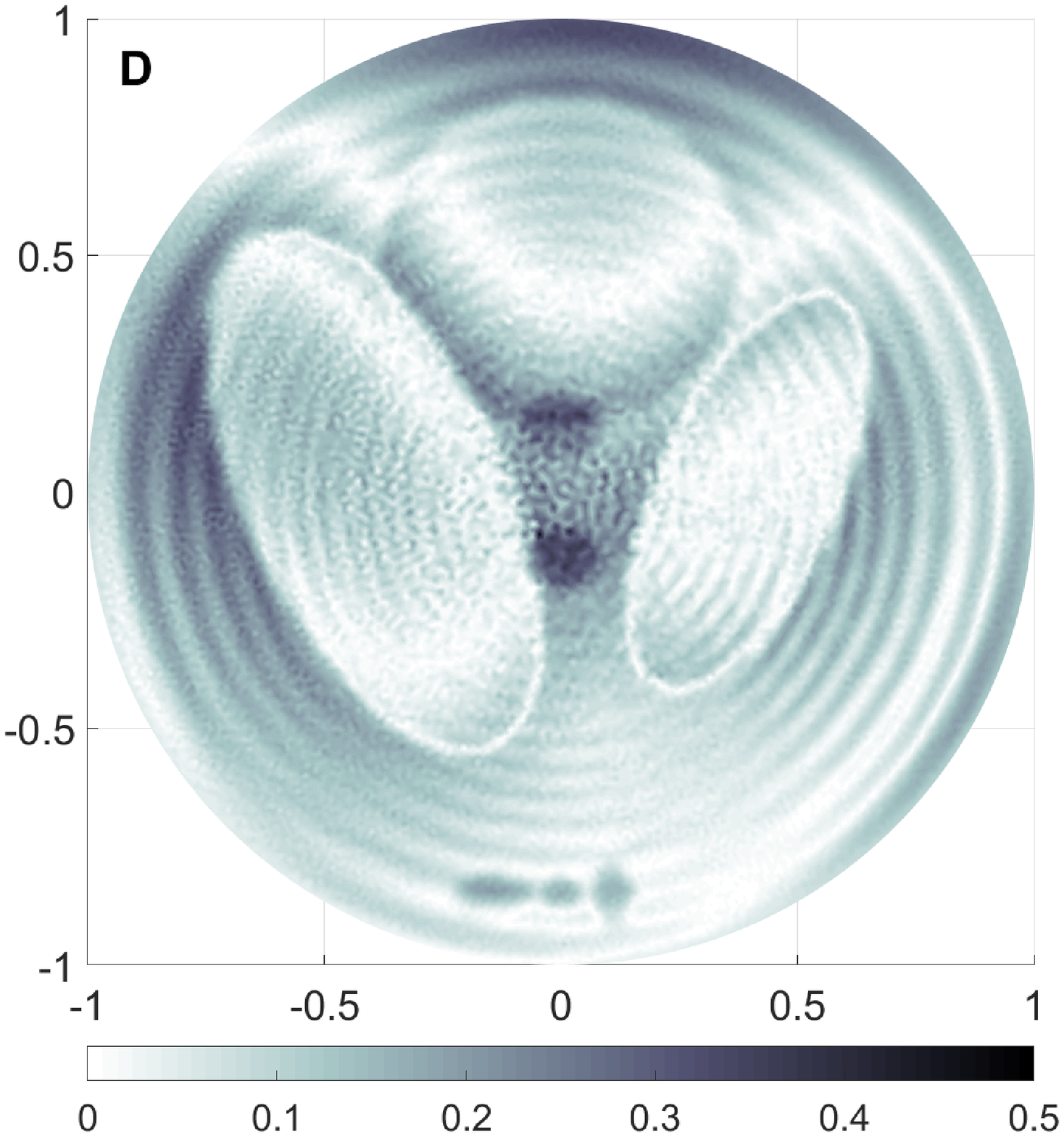}
  \vspace{-1em}
  \end{subfigure}
  \caption{Panel \textbf{A}: Reconstruction accounting for the measured data of the Fabry--Perot sensor as modeled by (\ref{eq:eff_measurement_simple}). Panel \textbf{B}: Reconstruction obtained using the overly idealized (but commonly assumed) model of measured data (\ref{eq:overly_ideal_measurement}). Panel \textbf{C}: Error profile for the reconstruction shown in panel \textbf{A}. Panel \textbf{D}: Error profile for the reconstruction shown in panel \textbf{B}.}
  \label{fig:Recon}
\end{figure}

%%%%%%%%%%%%%%%%%%%%%%%%%%%%%%%%%%%%%%%%%%%%%%%%%%%%%%%
%%%%%%%%%%%%%%%%   NEW SECTION   %%%%%%%%%%%%%%%%%%%%%%
%%%%%%%%%%%%%%%%%%%%%%%%%%%%%%%%%%%%%%%%%%%%%%%%%%%%%%%
\section{Conclusions}
\label{sec:conclusions}

We have developed a model for the type of measurements acquired by sensors based on the Fabry--Perot design. This model takes the form shown in (\ref{eq:eff_measurement_simple}). The validity of this expression is limited to small values for the thickness $h$ of the sensing film with respect to the wavelength of the pressure fields. This means that $h \ll \Csen/f$ where $f$ is the oscillatory frequency. For instance, \cite{Cox2007c} considered a Fabry--Perot polymer film of thickness $h = 40$ $\mu$m with compressional wave speed $\Csen = 2200$ m/s and a frequency range $1-15$ MHz. At the higher end of this range, the film thickness is about one fourth of the wavelength. Therefore, the proposed model would be valid for most of this frequency range.

Our mathematical model of the measurements captures the directional response observed experimentally \cite{Cox2007c,Guggenheim2017,Sheaff2014,Yoo2015,Zhang2008}. 
For instance, notice that for a pressure wave $p$ impinging the boundary $\Gamma$ in the normal direction, the sensor design fully captures the pressure field. However, for pressure waves traveling at other incidence angles, the sensor response may exhibit non-ideal behavior, such as vanishing response at critical angles, as shown in \Cref{fig:Response}. This is the mathematical description of the directivity associated with these ultrasound sensors. The incorporation of these features into reconstruction algorithms has been recognized as one of the challenges associated with improving photoacoustic inversion \cite{Cox2009a,Ellwood2014,Wang2011a,Xia2014}.

Using the model (\ref{eq:eff_measurement_simple}) for the measurements, we studied the solvability of the PAT problem and concluded that the problem is well-posed in the appropriate spaces and norms. See the precise statements in \cref{thm:main}. Following the analysis, a reconstruction algorithm was implemented based on the Landweber iteration. We carried out proof-of-concept numerical simulations to illustrate the reconstructions obtained from this method for synthetic data after discretization using FEM. For the chosen Shepp--Logan phantom, \Cref{fig:Recon} displays the results obtained by incorporating (Panel \textbf{A}) and by ignoring (Panel \textbf{B}) the model for the Fabry--Perot measurements. The respective error profiles are shown in Panels \textbf{C} and \textbf{D} of the same figure. Approximately, a $22\%$ relative error is added when the proposed model for the Fabry--Perot measurement is not incorporated in the reconstruction algorithm. We also highlight the qualitative difference between the error profiles from Panels \textbf{C} and \textbf{D} of \Cref{fig:Recon}. By ignoring the Fabry--Perot model, the error profile exhibits prominent artifacts over the entire image, especially near the detection boundary. These artifacts may be explained by the directivity response shown in \Cref{fig:Response}. By design, the proposed reconstruction algorithm accounts for the directivity response of these sensors leading to the removal of those artifacts.

Finally, we propose a couple of directions for future research that may improve or extend this work. It remains to study the well-posedness of the PAT problem for a Fabry--Perot measurement model that includes both the p-waves and s-waves in the elastic sensing film and backing substrate of the sensor. Such a model would incorporate the influence of the shear waves on the measurements as studied by Cox and Beard for plane waves \cite{Cox2007c}. Also, non-trivial directivity responses are not only induced by the Fabry--Perot sensor design, but also by piezoelectric detectors \cite{Beard2000b,Nuster2009,Nuster2019,Paltauf2017,Wilkens2007}. Therefore, analysis of the well-posedness for the PAT problem using piezoelectric measurements is also needed.

%%%%%%%%%%%%%%%%%%%%%%%%%%%%%%%%%%%%%%%%%%%%%%%%%%%%%%%
%%%%%%%%%%%%%%%%   NEW SECTION   %%%%%%%%%%%%%%%%%%%%%%
%%%%%%%%%%%%%%%%%%%%%%%%%%%%%%%%%%%%%%%%%%%%%%%%%%%%%%%
\section*{Acknowledgments}
The author would like to thank Texas Children's Hospital for its support and for the research-oriented environment provided by the Predictive Analytics Lab. 

%The author would like to thank the anonymous referees for carefully reviewing the manuscript and for
%providing helpful recommendations.

\bibliographystyle{siamplain}
\bibliography{InvProblemBiblio}

\end{document}